\documentclass[reqno]{amsart}[12pt]
\usepackage{fullpage}
\usepackage{amsmath, amssymb, latexsym, amsthm, amsfonts, mathrsfs,
	amscd,color, bbm, enumitem}
\usepackage{scalerel,stackengine}
\stackMath

\newcommand\reallywidehat[1]{%
	
	\savestack{\tmpbox}{\stretchto{%
			\scaleto{%
				\scalerel*[\widthof{\ensuremath{#1}}]{\kern-.6pt\bigwedge\kern-.6pt}%
				{\rule[-\textheight/2]{1ex}{\textheight}}
			}{\textheight}%
		}{0.5ex}}%
	\stackon[1pt]{#1}{\tmpbox}%
}
\numberwithin{equation}{section}

\theoremstyle{plain}
\newtheorem{theorem}{Theorem}
\newtheorem*{theorem*}{Theorem}

\newtheorem{corollary}[theorem]{Corollary}
\newtheorem{proposition}[theorem]{Proposition}

\theoremstyle{remark}
\newtheorem*{remark}{Remark}

\theoremstyle{definition}
\newtheorem*{definition}{Definition}

\newcommand{\Q}{{\mathbbm Q}}

\newcommand{\Z}{{\mathbbm Z}}

\newcommand{\cO}{{\mathcal O}}

\begin{document}
	
	\title{A finiteness theorem for abelian varieties \\ with totally bad reduction}
	
	\author{Plawan Das}

	\address{Chennai Mathematical Institute,  INDIA.}  
	\email{plawan@cmi.ac.in, d.plawan@gmail.com}
	
	\author{C.~S.~Rajan}
	
	\address{Tata Institute of Fundamental  Research, Homi Bhabha Road,
		Bombay - 400 005, INDIA.}  \email{rajan@math.tifr.res.in}

	
	\begin{abstract} 
		We show that up to potential isogeny, there are only finitely many abelian varieties of dimension $d$ defined over a number field $K$, such that for any finite place $v$ outside a fixed finite set $S$ of places of $K$ containing the archimedean places, it has either good reduction at $v$, or totally bad reduction at $v$ and good reduction over a quadratic extension of the completion of $K$ at $v$. 
	\end{abstract}
	\maketitle

	\section{Introduction}
	In \cite{Fa}, Faltings proved a conjecture of Shafarevich that given a number field $K$ and a finite set of places $S$ of $K$, there are only finitely many isogeny classes of abelian varieties of dimension $d$ over $K$ that have good reduction at the finite places of $K$ outside $S$. In this note, we consider some finiteness analogues for abelian varieties defined over $K$ satisfying potential good reduction outside $S$. 
	
	Two abelian varieties defined over a field $K$ are said to be {\em potentially isogenous} if they become isogenous over a finite extension of $K$. Suppose  $A$ is an abelian variety defined over a number field $K$. For a place $v$ of $K$, let $A_v=A\times_{Spec(K)} Spec(K_v)$ denote its localization at $v$, where $K_v$ is the completion of $K$ at $v$. We have, 
	
	\begin{theorem}\label{theorem:ellcurves}
		Let $K$ be a number field and $S$ be a finite set of places containing the archimedean places of $K$. Up to potential isogeny, there are only finitely many elliptic curves $E$ defined over $K$, such that for any place $v$ not in $S$, the elliptic curve $E_v$ has good reduction over a quadratic extension of $K_v$. 
	\end{theorem}
	
	For any field $F$, let $\bar{F}$ denote it's algebraic closure, and  
	$G_F=\mathrm{Gal}(\bar{F}/F)$ denote the absolute Galois group of $F$.
	We recall that for a quadratic character $\chi\in \mathrm{Hom}_{\text{cont}}(G_K,\{\pm 1\})\hookrightarrow \mathrm{Hom}_{\text{cont}}(G_K, \mathrm{Aut}_{\bar{K}}(A))$, the {\em quadratic twist} of $A$ by $\chi$, denoted by $A^{\chi}$, is an abelian variety defined over $K$ such that there is an $\bar{K}$-isomorphism $\psi:A\rightarrow A^{\chi}$ satisfying
	\[\psi^{-1}\psi^\sigma=[\chi(\sigma)],\quad \mbox{ for all $\sigma\in G_K$}. \] The foregoing result can be strengthened when the elliptic curves do not have complex multiplication i.e.,  $\mbox{End}_{\bar{K}}(E)=\Z$: 
	\begin{theorem}\label{theorem:noncm}
		Let $K$ be a number field and $S$ be a finite set of places containing the archimedean places of $K$. Up to quadratic twists, there are only finitely many isomorphism classes of non-CM elliptic curves $E$ defined over $K$, which acquires good reduction over some quadratic extension of $K_v$, for any place $v$ of $K$ not in $S$. 
	\end{theorem}
	
	\subsection{Abelian varieties} We now give a generalization of the above results to abelian varieties.  For a finite place $v$ of $K$, denote by $\mathcal{O}_v$ its ring of integers and  by $k_v$ its residue field.  Given a representation $\rho:\Gamma \to GL_n(F)$ of a group $\Gamma$ and $F$ any field, the algebraic monodromy group $G_{\rho}$ of $\rho$ is defined as the smallest algebraic subgroup defined over $F$ of $GL_n$, such that $G_{\rho}(F)$ contains the image $\rho(\Gamma)$. Equivalently, it is the algebraic closure inside $GL_n$ of the image $\rho(\Gamma)$ of $\rho$. 
	
	Suppose $A$ is an abelian variety defined over a number field $K$. Let  $\mathcal{A}_v$ denotes the N\'eron model of $A$ at $v$. This is a smooth group scheme of finite type over $\cO_v$, with generic fiber isomorphic to $A$ and representing the functor $X\mapsto \mbox{Hom}(X\times_{\cO_v}K_v, A_v)$ for $X$ a smooth scheme over $\cO_v$. Let $\tilde{A}_v$ denote the special fibre of the N\'eron model $\mathcal{A}_v$ and let $\tilde{A}_v^0$ denote its connected component containing identity. 
	
	\begin{definition} The abelian variety $A_v$ over $K_v$ is said to have {\em totally bad reduction} (resp. {\em purely additive reduction}) if $\tilde{A}_v^0$ is an affine group scheme (resp.  unipotent group scheme) over $k_v$. 
	\end{definition}
	\begin{remark}
		Suppose $A$ is an elliptic curve, then $A$ has good reduction at $v$ (resp. semi-stable reduction, additive reduction) if and only if  $\tilde{A}_v$ is an elliptic curve ( resp. $\tilde{A}_v^0\cong \mathbb{G}_m$, $\tilde{A}_v^0\cong \mathbb{G}_a$). Hence totally bad reduction is same as bad reduction and purely additive reduction is same as additive reduction.
	\end{remark}
	We now give a generalization of theorem $\ref{theorem:ellcurves}$ and theorem $\ref{theorem:noncm}$ to the case of abelian varieties.
	
	\begin{theorem}\label{theorem:general}
		Let $K$ be a number field and $S$ be a finite set of places containing the archimedean places of $K$. Fix a natural number $d$. Up to potential isogeny, there are only finitely many abelian varieties $A$ of dimension $d$ defined over $K$, such that for any place $v$ not in $S$, the abelian variety $A_v$ has either good reduction at $v$, or totally bad reduction at $v$ and good reduction over a quadratic extension of $K_v$. 
		
		If in addition, $\mbox{End}_{\bar{K}}(A)=\Z$,  then up to quadratic twists, there are only finitely many isomorphism classes of such abelian varieties defined over $K$. 
	\end{theorem}
	
	\begin{remark}
		One cannot expect the finiteness up to potential isogeny of elliptic curves defined over a number field which have potentially good reduction outside a given set of places $S$ of $K$. This can be seen as follows: it is known that given any $x\in K$, there exists an elliptic curve $E_x$ defined over $K$ whose $j$-invariant is $x$. The curve $E_x$ has potentially good reduction at a place $v$ not in $S$ if and only if its $j$-invariant $x$ is integral at $v$, i.e., $v(x)\geq 0$. The curve $E_x$ can be assumed to be non-CM, by taking $S$ to be non-empty and $x$ not integral at a place in $S$. There are infinitely many $x$ satisfying the above conditions. 
		
		Suppose there are only finitely many such curves up to potential isogeny. Arguing as in the proof of Theorem \ref{theorem:noncm} given below, we get that up to twisting by a quadratic character, there are only finitely many equivalence classes of such elliptic curves. Since twisting does not change the $j$-invariant, there are only finitely many possible $j$-invariants. This contradicts the infinitude of $j$-invariants given above. 
		
	\end{remark}
	
	\begin{remark}		It would be interesting to know the extent to which potential finiteness theorems can be proven. For example, do we have a finiteness result as above for abelian surfaces, which locally at a place $v$ outside $S$ has either good reduction or the connected component of its N\'{e}ron model at $v$ is an extension of an elliptic curve by $\mathbb{G}_a$?
		
		Does there exist potential finiteness if we work with fixed bounded degree finite extensions of number fields? For example, can we say that there are only finitely many isogeny classes of elliptic curves over a number field $K$ that have good reduction at all primes outside the primes lying over $S$ in some quartic extension of $K$?		
		
	\end{remark}

	\section{Elliptic Curves }
	We state a proposition giving alternate characterizations in the case of elliptic curves of the hypothesis in Theorem \ref{theorem:ellcurves}: 
	\begin{proposition}\label{Additive_example}
		Let $E$ be an elliptic curve over a  local field $K$ with bad reduction, where $K$ is a finite extension of $\mathbb{Q}_p$ and $\mathfrak{p}$ be the prime lying above $p$. Let $p \neq 2,3$ and  $\Delta$ be the discriminant of $E$ corresponding to a Weierstrass model of $E$. The following are equivalent:
		\begin{itemize}
			\item[(i)] The base change $E_L$ has good reduction for some quadratic extension $L / K$.
			\item[(ii)] The Galois group of the minimal extension of $K^\mathrm{unr}$ over which $E$ acquires good reduction, denoted by $\Phi_{\mathfrak{p}}$, is cyclic of order $2$, where $K^{\mathrm{unr}}$ denotes the maximal unramified extension of $K$.
			\item[(iii)] The Kodaira type of $E$ is $I_{0}^{*}$.
			\item[(iv)]   $v_{\mathfrak{p}}(\Delta)\equiv 6\pmod{12}$. If $\Delta_\text{min}$ denotes the minimal discriminant of $E$, then $v_{\mathfrak{p}}(\Delta_\text{min})=6$. 
		\end{itemize}
	\end{proposition}
	\begin{proof}
		First we show $(i)\implies (ii)\iff (iii)\iff (iv)$. It follows from the hypothesis that the extension $L$ over $K$ is ramified, since it's a quadratic extension, it becomes totally ramified. Let $L'$ be the compositum $LK^{\mathrm{unr}}$. We have $[L:K]=[L':K^{\mathrm{unr}}]$. Then by \cite[p. 312]{Se}, \cite[Theorem 8.2]{Si2}, we have
		$$[L:K]=2\implies\textrm{card}(\Phi_\mathfrak{p})=2\iff \text{Kodaira type}~I_{0}^{*}~( \text{or N\'eron type}~c_4 )\iff v_{\mathfrak{p}}(\Delta)\equiv 6\pmod{12}.$$
		
		$(ii)\implies(i):$ Let $M$ be a degree two extension of $K^{\mathrm{unr}}$, hence totally ramified. We can check that $M=K^{\mathrm{unr}}(\sqrt{\pi})$, $\pi$ is a uniformizer in $K$. As $\mathbb{Q}_p^{\mathrm{unr}}\subset K^{\mathrm{unr}}$, it follows that $|. |_p$ extends uniquely to the field $K^{\mathrm{unr}}$, and  it's ring of integers, $\mathcal{O}_{K^{\mathrm{unr}}}=\{x\in K^{\mathrm{unr}}~:~|x|_p\leq 1 \}$ is a local ring. All degree two totally ramified extensions of $K^{\mathrm{unr}}$ are obtained as $M=K^{\mathrm{unr}}(\sqrt{u\pi})$, for some $u\in {\mathcal{O}^ \times_{K^\mathrm{unr}}}$. Since $p\neq 2$, using Hensel's lemma to the Eisenstein polynomial $X^2-u$,  we can show that $\sqrt{u}\in {\mathcal{O}}^\times_{K^{\mathrm{unr}}}$. Hence we have $M=K^{\mathrm{unr}}(\sqrt{\pi})$ and we take $L=K(\sqrt{\pi})$.
	\end{proof}
	\begin{remark}
		Let $p>3$ and $K$ is a finite extension of $\mathbb{Q}_p$. Suppose $A$ be an abelian variety defined over $K$ with potential good reduction. Let $M$ be the minimal degree $n$ extension of $K^{\mathrm{unr}}$ such that the base change $E_M$ has good reduction (see \cite[Corollary 3, p. 498]{ST}), where $(n,p)=1$. Now as $K^\mathrm{unr}/K$ is  Galois extension and $M$ is a unique subextension of $\bar{K}$, we have $M/K$ Galois. We have a exact sequence
		\begin{equation*}
			1\longrightarrow \mathrm{Gal}(M/K^{\mathrm{unr}})\longrightarrow
			\mathrm{Gal}(M/K)\longrightarrow \mathrm{Gal}(K^{\mathrm{unr}}/K) \longrightarrow 1,
		\end{equation*}
		which splits as $\mathrm{Gal}(K^{\mathrm{unr}}/K)$ is free profinite group. Let $L$ be the subextension of $M$ fixed by the image of $\iota:\mathrm{Gal}(K^{\mathrm{unr}}/K)\hookrightarrow \mathrm{Gal}(M/K)  $. So, $L$ is finite totally ramified extension of degree $n$ and we have $M=L.K^{\mathrm{unr}}$.
	\end{remark}
	
	Now using the above proposition we get the following corollary of theorem \ref{theorem:noncm}:
	\begin{corollary}
		Let $K$ be a number field and $S$ be a finite set of primes containing the archimedean primes of $K$ together with primes lying above $2$ and $3$. Up to quadratic twists, there are only finitely many isomorphism classes of non-CM elliptic curves $E$ defined over $K$ with $v_\mathfrak{p}(\Delta_{min})=0$ or $6$ for any place $\mathfrak{p}$ of $K$ not in $S$.  
	\end{corollary}

	\section{Proof of Theorem \ref{theorem:general}}
	
	Let $A$ be an abelian variety defined over a number field $K$. We denote the $\ell$-adic Tate module attached to $A$ by $T_{\ell}(A)=\varprojlim_nA[\ell^n]$, where $A[\ell^n]$ is the group of elements of order  $\ell^n$ in $A(\bar{K})$. This carries with it a continuous action $\rho_{A,\ell}$ of $G_K$. Let $V_\ell(A):=\Q_\ell\otimes_{\Z_\ell} T_{\ell}(A)$. 
	
	For a place $v$ of $K$, denote by $I_v$ the inertia subgroup of $G_{K_v}$. We observe that under our hypothesis, the inertial invariants are trivial: 
	\begin{proposition}
		Let $v$ be a finite place of $K$ at which the abelian variety $A_v$ has totally bad reduction and has potentially good reduction. 
		
		Then for any rational prime $\ell$ coprime to the residue characteristic of $K_v$, the subspace $V_{\ell}(A_v)^{I_v}$ of inertial invariants of the Tate module is trivial. 
	\end{proposition} 
	
	\begin{proof}
		We follow the arguments given in (\cite{ST}, see also \cite[Theorem 10.2, Chapter IV]{Si2} and \cite[Corollary 1.10]{LO}) to prove this result. By (\cite[Lemma 2]{ST}), 
		\[\mbox{Hom}(\Z/\ell^n\Z, A_v(\bar{K}_v))^{I_v}\simeq 
		\mbox{Hom}(\Z/\ell^n\Z, \tilde{A}_v(\bar{k_v})),\]
		where $k_v$ is the residue field of $K_v$. Now by  (\cite{ST}[Remark (1), Section 3, p.500]), we have that $A_v$ has purely additive reduction at $v$. Since the connected component $\tilde{A}_v^0$ of the N\'{e}ron model $\tilde{A}_v$ is unipotent, 
		\[\mbox{Hom}(\Z/\ell^n\Z, \tilde{A}_v^0(\bar{k_v}))=0.\]
		Hence, $\mbox{Hom}(\Z/\ell^n\Z, \tilde{A}_v(\bar{k_v}))$ is a finite abelian group of order dividing $c$, where $c$ is the number of connected components of $\tilde{A}_v$. Taking the limit over $n$ and tensoring with $\Q_{\ell}$ proves the proposition. 
	\end{proof}
	
	We observe now that the inertia at a place of $K$ outside $S$ acts by scalars:
	\begin{corollary} Let $v$ be a finite place of $K$ at which the abelian variety $A_v$ has totally bad reduction and acquires good reduction over a quadratic extension of $K_v$. Assume further that the residue characteristic of $K_v$ is not two. 
		
		Then for any rational prime $\ell$ coprime to the residue characteristic of $K_v$, the inertia group $I_v$ acts via scalars $\{\pm 1\}$ on the Tate module $V_{\ell}(A_v)$. 
	\end{corollary}
	\begin{proof}
		Suppose $L_w$ is a quadratic extension of $K_v$ at which $A_v$ acquires good reduction, and let $I_w$ denote its inertia subgroup. For an abelian variety $A$ over $K$, Serre and Tate (\cite{ST}), 
		extending the work of N\'{e}ron, Ogg, and Shafarevich on elliptic curves, showed that  $A$ has good reduction at a finite place $v$ of $K$, if and only if the representation of $G_K$ on the Tate module $V_\ell(A)$ attached to $A$ is unramified at $v$, for all $\ell$ not equal to the residue characteristic at $v$. 
		
		It follows from the hypothesis that the extension $L_w$ over $K_v$ is ramified, and  that the action of the inertia group $I_v$ on $V_\ell(A)$ is trivial restricted to the subgroup $I_w$ of index two in $I_v$.
		
		By the above proposition,  $V_\ell(A_v)^{I_v}=0$, i.e., there exists no invariant vector for $I_v$ action on $V_\ell(E)$. So the action of  $I_v/I_w\cong \Z/{2\Z}$ on $V_\ell(A_v)$ can't fix any non zero vector, i.e., the non trivial element $\sigma$ of $I_v/I_w$ acts by $-1$ on $V_\ell(A_v)$. 
	\end{proof}
	\begin{remark}
		$A_v^\chi$ defined over $K_v$ has good reduction, where  $\chi$ is the quadratic character corresponding to $L_w/K_v$. This generalizes Kida (\cite[ Proposition 2.2]{Ki}).
		
	\end{remark}
	
	\subsection{Potential finiteness} In \cite{Fa}, Faltings proved Tate's conjecture on homomorphisms between abelian varieties $A, B$ defined over $K$: 
	\[ \mbox{Hom}(A,B)\otimes \Q_{\ell}\simeq \mbox{Hom}_{G_K}(V_{\ell}(A), V_{\ell}(B)).\]
	From this, Shafarevich's conjecture was deduced by showing a finiteness criteria determining the isomorphism class of a $\ell$-adic representation: given $K, S$ and $n$, there exists a finite set of places $T$ disjoint from $S$, such that the traces of the Frobenius classes at $T$ serves to determine uniquely a representation $\rho:G_K\to GL_n(\Q_{\ell})$ that is unramified outside $S$. 
	
	In (\cite{DR}), a generalization of the finiteness criteria is given, where instead of the assumption that the representations are unramified outside $S$, it is assumed that the inertia outside $S$ can act by scalar endomorphisms (\cite[Theorem 7]{DR}). 
	
	The proof of Theorem \ref{theorem:ellcurves} and the first part of Theorem \ref{theorem:general}, follows from the above corollary together with the following extension of Faltings' finiteness criteria proved in (\cite[Corollary 11]{DR}): 
	\begin{theorem}[Corollary 11, \cite{DR}]
		Let $K$ be a number field and $S$ be a finite set of non-archimedean
		places of $K$. 
		
		Then there are, up to potential isogeny, only finitely many Abelian varieties of dimension $g$ defined over $K$, such that the inertia at a place $v$ not in $S$ acts by scalars on a Tate module  $V_\ell(A)$ for $\ell$ coprime to the residue characteristic at $v$. 
	\end{theorem}

	\subsection{Twist equivalence} Suppose $\mbox{End}_{\bar{K}}(A)=\Z$, we now give two proofs that up to twist equivalence  there are only finitely many such abelian varieties satisfying the hypothesis of Theorem \ref{theorem:general}. For such abelian varieties, the group of isogenies of $A$ onto itself can be identified with $\mbox{End}_{\bar{K}}(A)\otimes \Q\simeq \Q$. 
	
	Suppose $A, A'$ are abelian varieties defined over $K$ as above and $\phi$ is an isogeny from $A'$ to $A$ defined over $\bar{K}$. Let $\phi^*$ denote the dual isogeny $A\to A'$. We have for some natural number $d$, the degree of $\phi$, 
	\[ \phi^*\circ \phi=[d], \]
	where $[d]$ denotes the endomorphism of $A'$ given by multiplication by $d$. 
	For $\sigma\in G_K$, define an element of $\mbox{End}_{\bar{K}}(A)\otimes \Q\simeq \Q$, by 
	\[ c_{\sigma}=\frac{1}{d}\phi^{\sigma}\circ \phi^*.\]
	It can be checked that $\sigma\mapsto c_{\sigma}$ defines a $1$-cocycle of $G_K$ with values in $\Q^*$:
	\[\begin{split}
		c_{\sigma\tau}&=\frac{1}{d}\phi^{\sigma\tau}\circ \phi^*\\
		&=\frac{1}{d^2}\phi^{\sigma\tau}\circ (\phi^*)^{\tau}\circ\phi^{\tau}\circ\phi^*\\
		&=\left(\frac{1}{d}\phi^{\sigma}\circ \phi^*\right)^{\tau}\circ \frac{1}{d}\phi^{\tau}\circ \phi^*\\
		&= c_{\sigma}^{\tau}c_{\tau}.
	\end{split}
	\]
	The Galois group $G_K$ acts trivially on $\Q^*$. Hence, 
	\[ H^1(G_K, \Q^*)=\mbox{Hom}(G_K, \Q^*)\simeq \mbox{Hom}(G_K, \{\pm 1\}).\]
	This shows that $A$ and twist of $A'$ by a quadratic character are isogenus. 
	
	By the finiteness up to potential isogeny given above, we can choose some finitely many isogeny classes of abelian varieties defined over $K$, say given by $A_1, \cdots, A_r$ such that any abelian variety $A'$ satisfying the hypothesis of the theorem is potentially isogenous to some $A_i$. 
	
	Hence there exists a quadratic character $\chi$ of $G_K$ such that $A'^{\chi}$ belongs to the isogeny class of $A_i$. The theorem follows from the fact that there are only finitely many abelian varieties up to equivalence in any isogeny class. 
	
	\noindent	{\it Alternate proof:} By the validity of Tate's conjecture proved by Faltings and Schur's lemma, the Galois group $G_K$ acts absolutely irreducibly on the Tate module $V_\ell(A)$ (for some or all $\ell$) precisely when $\mbox{End}_{K}(A)=\Z$. The condition $\mbox{End}_{\bar{K}}(A)=\Z$ amounts to saying that the connected component of the algebraic monodromy group of $G_K$ acts absolutely irreducibly on the Tate module $V_\ell(A)$ for one or equivalently for all $\ell$. 
	
	For any $\ell$, the Galois representations $\rho_{A, \ell}$ and $\rho_{A',\ell}$ are potentially isomorphic by the first part of the theorem. Since the connected component of the algebraic monodromy group acts absolutely irreducibly, it follows that there is a unique finite order character $\chi_{\ell}:G_K\to \bar{\Q}_{\ell}^*$, such that 
	\[ \rho_{A',\ell}=\rho_{A,\ell}\otimes \chi_{\ell},\] and the character $\chi$ can be seen to have values in $\Q_{\ell}^*$ (see \cite[Proposition 12]{DR}). By uniqueness, it is independent of $\ell$. Hence it has to take values in $\{\pm 1\}$. We then use the fact that there are only finitely many abelian varieties up to equivalence in any isogeny class to conclude the proof of finiteness up to twists.

	\subsection*{Acknowledgements} The first author thanks the School of Mathematics, Tata Institute of Fundamental Research, Mumbai and UM-DAE Centre for Excellence in Basic Sciences, Mumbai for their support. The research of the first author is supported by a DST-INSPIRE Fellowship IF150349.


\begin{thebibliography}{XXXXX}
		
		
		\bibitem[DR]{DR} P. Das,  C. S. Rajan, {\em Finiteness theorems for potentially  equivalent  Galois representations: extension of Faltings' finiteness criteria}, to appear in Proc. Amer. Math. Soc., doi:10.1090/proc/15856.	
		
		
		\bibitem[Fa]{Fa}
		G. Faltings,{\it Endlichkeitss\"{a}tze f\"{u}r abelsche Variet\"{a}ten \"{u}ber Zahlk\"{o}rpern}, Invent. Math. {\bf 73} (1983), no.~3, 349--366. MR0718935
		
		
		\bibitem[FW]{FW}
		G. Faltings\ et al., {\it Rational points}, third edition, Aspects of Mathematics, E6, Friedr. Vieweg \& Sohn, Braunschweig, 1992. MR1175627
		
		
		\bibitem[Ki]{Ki}
		M. Kida, {\it Potential good reduction of elliptic curves}, J. Symbolic Comput. {\bf 34} (2002), no.~3, 173--180. MR1935076
		
		
		\bibitem[LO]{LO}
		H. W. Lenstra, Jr.\ and\ F. Oort, Abelian varieties having purely additive reduction, J. Pure Appl. Algebra {\bf 36} (1985), no.~3, 281--298. MR0790619
		
		\bibitem[Se]{Se}
		J.-P. Serre, Propri\'{e}t\'{e}s galoisiennes des points d'ordre fini des courbes elliptiques, Invent. Math. {\bf 15} (1972), no.~4, 259--331. MR0387283
		
		\bibitem[Si1]{Si1}
		J. H. Silverman, {\it The arithmetic of elliptic curves}, second edition, Graduate Texts in Mathematics, 106, Springer, Dordrecht, 2009. MR2514094
		
		
		\bibitem[Si2]{Si2}
		J. H. Silverman, {\it Advanced topics in the arithmetic of elliptic curves}, Graduate Texts in Mathematics, 151, Springer-Verlag, New York, 1994. MR1312368
		
		\bibitem[ST]{ST}
		J.-P. Serre\ and\ J. Tate, Good reduction of abelian varieties, Ann. of Math. (2) {\bf 88} (1968), 492--517. MR0236190
		
		
		
	\end{thebibliography}
\end{document}